\documentclass[11pt]{amsart}

\usepackage[T1]{fontenc}
\usepackage{times}
\usepackage{amssymb,epsfig,verbatim} 
\usepackage{calligra,mathrsfs}
\theoremstyle{plain}
\usepackage[T1]{fontenc}
\usepackage[utf8]{inputenc}
\usepackage[english]{babel}
\usepackage{amssymb}
\usepackage{amsthm}
\usepackage{graphics}
\usepackage{amsmath}
\usepackage{amstext}
\usepackage{multirow}
\usepackage{enumitem}

\usepackage{hyperref}
\usepackage{graphicx}
\usepackage{amsmath}
\usepackage{amssymb}
\usepackage{amsthm}
\usepackage{amstext}
\usepackage{epstopdf}
\usepackage{mathrsfs}
\usepackage{amscd}
\usepackage{tikz}
\usetikzlibrary{shapes.geometric, arrows,matrix}
\usetikzlibrary{cd}
\AtBeginDocument{%
   \def\MR#1{}
}

\newtheorem{thm}{Theorem}[section]
\newtheorem{dfn}[thm]{Definition}
\newtheorem{lemma}[thm]{Lemma}
\newtheorem{prop}[thm]{Proposition}
\newtheorem{cor}[thm]{Corollary}

\newtheorem{example}[thm]{Example}

\theoremstyle{remark}
\newtheorem{remark}[thm]{Remark}

\newcommand{\mb}{\mathbb}

\newcommand{\mc}{\mathcal}

\newcommand{\C}{\mb C}
\newcommand{\Pj}{\mb P}

\newcommand{\F}{\mc F}
\newcommand{\G}{\mc G}

\newcommand{\germe}{\mathcal{O}}
\newcommand{\til}{\widetilde}
\newcommand{\D}{\mathbb{D}}

\newcommand\restr[2]{{
  \left.\kern-\nulldelimiterspace 
  #1 
  \vphantom{\big|} 
  \right|_{#2} 
  }}

\DeclareMathOperator{\Aut}{Aut}

\DeclareMathOperator{\Div}{Div}

\DeclareMathOperator{\kod}{kod}

\DeclareMathOperator{\HomSheaf}{\mathscr{H}\text{\kern -3pt {\calligra\large om}}\,}

\newcommand{\ie}{{\it{i.e. }}}

\numberwithin{equation}{section}

\addtocounter{section}{0}             
\numberwithin{equation}{section}       

\begin{document}
\title{Transversely projective structures on smooth foliations on surfaces}

\author[G. Fazoli]{Gabriel Fazoli}
\address{IMPA, Estrada Dona  Castorina, 110 -- 22460-320, Rio de Janeiro, RJ, Brazil}
\email{gabrielfazoli@gmail.com}

\author[C. Melo]{Caio Melo}
\address{IMPA, Estrada Dona  Castorina, 110 -- 22460-320, Rio de Janeiro, RJ, Brazil}
\email{caio8.melo@gmail.com}

\author[J. V. Pereira]{Jorge Vitório Pereira}
\address{IMPA, Estrada Dona  Castorina, 110 -- 22460-320, Rio de Janeiro, RJ, Brazil}
\email{jvp@impa.br}



\begin{abstract}
    Brunella's classification implies that every smooth foliation on a compact complex surface admits a singular transversely projective structure. However, Biswas and Dumitrescu's recent work shows that certain foliations on compact complex surfaces, despite possessing a singular transversely projective structure, do not admit a regular transversely projective structure. Here, we describe the smooth foliations on compact complex surfaces which fail to possess a regular transversely projective structure.
\end{abstract}

\maketitle
\date{}

\section{Introduction}

In \cite{biswas2024holomorphic}, Biswas and Dumitrescu employed a dimension counting argument to establish that a sufficiently general smooth foliation on the product of elliptic curves does not admit a transversely projective structure. Inspired by their work, our pourpose here is to describe all smooth holomorphic foliations on compact complex surfaces which do not admit a transversely projective structure.

Naturally, we build on Brunella's classification of smooth foliations on compact complex surfaces.
According to \cite{zbMATH01088813}, if $X$ is a compact K\"ahler surface and $\F$ is a smooth foliation on it then
$\F$ is a fibration, or $\F$ is transverse to the fibers of a $\mathbb{P}^1$-bundle, \ie $\F$ is a Riccati foliation,
or $\F$ is transverse to the general fiber of an elliptic fibration, \ie $\F$ is a turbulent foliation, or
$\F$ is a linear foliation on a complex torus, or  $\F$ is a  transversely hyperbolic foliation with dense leaves.
Moreover, if $X$ is a non-K\"ahler compact surface carrying a smooth foliation $\F$ then $X$ is a Hopf surface or an Inoue surface
and $\F$ is any foliation on them. 

A careful look at Brunella's list just present reveals that besides  fibrations, turbulent foliation, and foliations on Hopf surfaces, every other smooth foliation on a compact complex surface admit a natural transversely projective structure.

As already pointed out by Biswas and Dumitrescu, turbulent foliations are not necessarily transversely projective. In Section \ref{S:turbulent} below, we characterize which turbulent foliations admit transverse projective structures, see for instance Theorem \ref{T: turbulent is transversely projective if and only if (A) or (M)}. 

Curiously, there are fibrations on Hopf surfaces which do not admit a transversely projective structure. We show that this phenomenon does not happen when the surface is K\"ahler and compact, see Theorem \ref{T: fibration in Kahler is transversely projective}. We also obtain an analogue statement for fibrations on projective manifolds of dimension three, see Theorem \ref{T:fibration on 3folds} for a precise statement.  Although the general foliation on a Hopf surface is not turbulent, as it turns out, their analysis is very similar to the analysis of turbulent foliations as explained in Section \ref{S:Hopf}.

\subsection*{Acknowledgements} We are indebt to Sorin Dumitrescu for bringing the problem addressed in this article to our attention and for correspondence regarding it. Furthermore, the authors express their appreciation for the financial support provided by:  CAPES-COFECUB program (project number: Ma 1017/24, Géométrie des équations différentielles et des variétés algébriques), funded by the French Ministry for Europe and Foreign Affairs, the French Ministry for Higher Education and Research, and CAPES; and CNPq Projeto Universal 408687/2023-1 "Geometria das Equações Diferenciais Algébricas". Pereira  acknowledges the support  from CNPq (Grant number 301683/2019-0), and FAPERJ (Grant number E26/200.550/2023). Fazoli acknowledges the support of FAPERJ (Grant number E26/201.353/2023). Melo acknowledges the support of CAPES-PROEX.

\section{Transversely projective foliations}

\subsection{Foliations} A smooth holomorphic foliation $\F$ on a complex manifold $X$ is determined by an involutive locally free subsheaf $T_{\F}$ of $T_X$ (the tangent sheaf/bundle of $\F$) with cokernel $T_X/T_{\F}$ also locally free. The cokernel $T_X/T_{\F}$ will be denoted by $N_{\F}$ and is the normal bundle of $\F$. The rank of $T_{\F}$ is the dimension of the foliation, while the rank of $N_{\F}$ is the codimension of the foliation. The dual of $N_{\F}$ is the conormal bundle of $\F$ and will be denoted by $N^*_{\F}$, and the dual of $T_{\F}$ is the cotangent bundle of $\F$ and will be denoted by $\Omega^1_{\F}$.

\subsection{Transversely projective structures} We recall one of the many equivalent definitions of transversely projective foliations, see for instance \cite{LORAY_2007}.

\begin{dfn}\label{D: definition by the bundle}
    A codimension one foliation $\F$ on a complex manifold $X$ is a (regular) transversely projective foliation if there exists
    \begin{enumerate}
        \item a $\Pj^1$-bundle $\pi: P \rightarrow X$  over $X$;
        \item a codimension one smooth foliation $\G$ on $P$ everywhere transverse to the fibers of  $\pi$; and
        \item a holomorphic section $\sigma: X \rightarrow P$ transverse to $\G$ at every point, such that $\F = \sigma^* \G$. 
    \end{enumerate}
    The datum $\mathcal{P} = (\pi , \G , \sigma)$ is called a (smooth) transversely projective structure of $\F$.
\end{dfn}

Two transversely projective structures $\mathcal P = (\pi:P \to X,\G,\sigma)$ and $\mathcal P' = (\pi':P'\to X,\G',\sigma')$ for the same
foliation $\F$ are equivalent if there exists an invertible morphism  $\Phi: P \to P'$ such that $\pi'\circ \Phi= \mathrm{Id}$, $\Phi^* \G' = \G$ and $\Phi^* \sigma' = \sigma$.

If $X$ is  a curve and $\F$ is the foliation with $T_{\F} = 0$ and $N_{\F} = T_X$ then a transversely projective structure for
$\F$ is nothing but a projective structure for the Riemann surface $X$ as explained in  \cite[Section 1.5]{loray-martin-09-MR2647972}.

The same foliation may admit several non-equivalent projective structures.

\begin{example}[Linear foliations on complex tori]\label{E:tori} If $X$ be a compact complex torus and $\F$ be a codimension one foliation defined by a non-zero section $\omega \in H^0(X,\Omega^1_X)$. For any $\lambda \in \mathbb C$, we may consider the foliation $\mathcal G_{\lambda}$ on the trivial $\mathbb P^1$-bundle $P= X \times \mathbb P^1$ defined by the 
    meromorphic $1$-form $d \log z - \omega + z \lambda \omega$, where $z$ is a projective coordinate  on $\mathbb P^1$. If $\sigma : X \to P$ is the section $\sigma(x) = (x, 1)$  then $\F = \sigma^* \mathcal G_\lambda$. In this way, we obtain a family of transversely projective structures $(P, \mathcal G_{\lambda}, \sigma)$ for $\F$ parameterized by $\lambda \in \C$. Moreover, if $\lambda \neq \lambda'$ then the corresponding transversely projective structures are non-equivalent.
\end{example}

The set of projective structures on a Riemann surface $X$ is parameterized by the affine space $H^0(X, (\Omega^1_X)^{\otimes 2})$, see \cite{loray-martin-09-MR2647972}. We recall below  an analogue for transversely projective foliations of this classical fact.

\begin{prop}\label{P:estruturas}
    Let $\F$ be a smooth codimension one foliation on a complex manifold $X$. The equivalence classes of transversely projective structures
    of $\F$ is parameterized by the kernel of morphism
    \[
        H^0(X, (N^*_{\F})^{\otimes 2}) \to  H^0(X,  (N^*_{\F})^{\otimes 2}  \otimes \Omega^1_{\F} )
    \]
    induced by Bott's partial connection. In particular, if $h^0(X,(N^*_{\F})^{ \otimes 2})=0$ then $\F$ admits at most one transversely projective structure.
\end{prop}

We refer to \cite{biswas2024holomorphic} for a proof of Proposition \ref{P:estruturas}.

\subsection{Distinguished local first integrals}
Alternatively, a transversely projective structure for a foliation $\F$ can be defined in terms of distinguished local first integrals.

\begin{dfn}\label{D: transversal projective structure}
    Let $X$ be a complex manifold. Let $\F$ be a smooth foliation on $X$. A transversely projective structure for $\F$ is a collection
    $\{f_i:U_i \rightarrow V_i \subset \mathbb{P}^1\}_{i \in I}$ of holomorphic submersions $f_i$ such that
    \begin{enumerate}
        \item the collection of domains $U_i$ forms an open cover of $X$;
        \item for every $i \in I$, the submersion $f_i$ define $\restr{\F}{U_i}$, that is
        \[
            \restr{T_{\F}}{U_i} = \ker df_i : \restr{T_{X}}{U_i} \to f_i^* T_{\mathbb P^1} ; \text{and}
        \]
        \item whenever $U_i \cap U_j \neq \emptyset$, there exists $f_{ij}\in \mathrm{Aut}(\mathbb{P}^1)$ such that
        \[
            \restr{f_i}{U_i \cap U_j} = f_{ij} \circ \restr{f_j}{U_i \cap U_j}\, .
        \]
    \end{enumerate}
\end{dfn}

Definitions \ref{D: definition by the bundle} and \ref{D: transversal projective structure} are equivalent. To pass from the first to the second, one starts with a sufficiently fine open covering $\{ U_i\}$ of $X$, consider $\Pj^1$-valued submersions $g_i : \pi^{-1}(U_i) \to \Pj^1$ which are first integrals for $\restr{\G}{\pi^{-1}(U_i)}$ (they exist and any two differ by left composition by an element of $\Aut(\Pj^1)$, because of the transversality between $\G$ and $\pi$), and consider the first integrals for $\restr{\F}{U_i}$ defined by $g_i \circ \sigma$. Reciprocally, to pass from Definition \ref{D: transversal projective structure} to Definition \ref{D: definition by the bundle} one constructs the $\Pj^1$-bundle $P$ by gluing $U_i\times \Pj^1$ with $U_j\times \Pj^1$ using the transition function $f_{ij}$ and interpret the first integral $f_i$ as a section of the $\Pj^1$-bundle. For further details on the equivalence of the two definitions, the reader can consult \cite{zbMATH00050350}.

\subsection{Canonical projective structures on Riemann surfaces}
Specializing Definition \ref{D: transversal projective structure} to codimension one foliations on curves (that is, curves), one recovers the definition of a projective structure on a Riemann surface in terms of an atlas with transition functions taking values in $\Aut(\Pj^1)$.

\begin{example}[Canonical projective structure on a Riemann surface] \label{Ex: Canonical projective structure} 
    Every Riemann surface $X$ admits a projective structure induced by the uniformization theorem. Moreover, such projective structure is invariant by the automorphism group of $X$.
\end{example}
\begin{proof}
    Let $X$ be a Riemann surface and $\pi: \Tilde{X} \rightarrow X$ its universal cover. According to the uniformization theorem, $\Tilde{X}$ is either $\mathbb{P}^1$, $\mathbb{C}$ or $\D$. Let $\iota: \Tilde{C} \hookrightarrow \Pj^1$ be the natural inclusion. By considering all local sections of $\pi$ and composing with $\iota$, we obtain a collection of embeddings $f_i: U_i \rightarrow \Pj^1$. The collection $\{f_i : U_i \to f_i(U_i) \subset \mathbb P^1\}$ is a projective structure for $X$. The invariance by the automorphism group of $X$ is clear since $\Aut(\D)$ and $\Aut(\C)$ are subgroups of $\Aut(\Pj^1)$.
\end{proof}

Following \cite{loray-martin-09-MR2647972}, we will call the projective structure given by Example \ref{Ex: Canonical projective structure}, the canonical projective structure of $X$. Example \ref{Ex: Canonical projective structure} has two natural counterparts for codimension one foliations.

\begin{example}[Smooth fibrations]
    Let $B$ be a Riemann surface. Let $\F$ be a foliation defined by a smooth fibration $f :X \to B$, that is, every fiber of $f$ is smooth and none of them is multiple. Then, every projective structure for $B$ induces a smooth transversely projective structure for $\F$, given as the pullback of the projective atlas. Since, by  Example \ref{Ex: Canonical projective structure}, every Riemann surface admits a smooth projective structure, then $\F$ is transversely projective.
\end{example}

\begin{example}[Foliations transverse to smooth fibrations]\label{E:transverse to fibrations}
    Let $\F$ be a smooth codimension one foliation on a complex manifold $X$. Assume $X$ is endowed with a proper smooth fibration $\pi:X \to B$ with one-dimensional fibers such that $\F$ is everywhere transverse to the fibers of $\pi$. Then $\F$ is transversely projective.
\end{example} 
\begin{proof}
    The transversality between $\F$ and $\pi$ impose strong restrictions on the foliation $\F$ and on the fibration $\pi$ as explained in \cite[Chapter 5]{zbMATH03907374} which we summarize. Lifting paths on the basis $B$ along leaves of $\F$ one verifies  that any two fibers of $\pi$ are isomorphic. Let us choose one of them, say $F_b = \pi^{-1}(b)$.
    Moreover, if $\Tilde B$ is the universal covering of $B$ and $\Tilde X$ is the fiber product $X \times_{B} \Tilde B$ then $\Tilde X$ is biholomorphic to the product $\Tilde B \times F_b$ in such a way that $\Tilde \F$, the pull-back of $\F$  to $\Tilde X$, is the foliation defined the projection on the second the factor. The original manifold $X$ can be recovered as a quotient of $\Tilde X$ by  $\pi_1(B,b)$ acting on $\Tilde B \times F$ through deck transformations on the first factor and through a representation $\rho : \pi_1(B,b) \to \Aut(F_b)$.

    The canonical projective structure on $F_b$ defines a projective structure on $\Tilde \F$. Since the canonical projective structure on $F_b$ is invariant by automorphisms, the projective structure for $\Tilde \F$ is invariant by the action of $\pi_1(B,b)$ and therefore descends to a projective structure on $\F$.
\end{proof}

\subsection{Monodromy representation} Given a transversely projective structure $( \pi:P \to X, \G, \sigma)$  for a foliation $\F$ on a complex manifold $X$ and a point $x \in X$, the lifting of paths on $X$ to the leaves of the foliation $\G$ defines  an anti-representation
\[
    \rho : \pi_1(X,x) \to \Aut(\mathbb P^1) \, ,
\]
which completely determines the isomorphism class of the $\mathbb P^1$-bundle $P$. It can also be obtained through the analytic continuation
of the distinguished first integrals for the projective structure.

If one considers the étale covering $p : Y \to X$ of $X$ determined by $\ker \rho$, one obtains that $p^* \F$, the pull-back of $\F$ to $Y$,
is defined by a submersion $f:Y \to \mathbb P^1$ which is covariant for the action of $\pi_1(X,x)/\ker \rho = \rho(\pi_1(X,x))$ on $Y$ that is
\[
    f(\gamma \cdot y ) = \rho(\gamma)(f(y))
\]
for every $\gamma \in \pi_1(X,x)/\ker \rho$.

\subsection{Singular projective structures} Although our focus is on smooth transversely projective structures, it seems appropriate
to recall the definition of singular transversely projective structures.

\begin{dfn}\label{D:singular structures}
    A codimension one foliation $\F$ on a complex manifold $X$ is a singular transversely projective foliation if there exists
    \begin{enumerate}
        \item a $\Pj^1$-bundle $\pi: P \rightarrow X$  over $X$; and
        \item a codimension one foliation $\G$ on $P$ transverse to the generic fiber of  $\pi$; and
        \item a meromorphic section $\sigma: X \dashrightarrow P$ generically transverse to $\G$
    \end{enumerate}
    such that $\F = \sigma^* \G$. The datum $\mathcal{P} = (\pi , \G , \sigma)$ is called a singular transversely projective structure of $\F$.
\end{dfn}

It is well-known, and easy to check using Brunella's classification, that every smooth foliation on a compact complex surface admits a singular transversely projective structure. For more about singular transversely projective foliations/structures, we refer to \cite{Scardua1997} (the first work to consider the concept), \cite{LORAY_2007} for a detailed discussion of the above definition, and \cite{loray2016representations} for a description of the structure of codimension one singular transversely projective foliations on projective manifolds of arbitrary dimension.

\section{Fibrations}
Let $X$ be a complex manifold and let $B$ be a smooth curve. We say that a proper fibration $f : X \to B$ is quasi-smooth if the support of every fiber of $f$ is smooth. In other words, for every $b \in B$, the divisor  $f^*b \in \Div(X)$ is a positive multiple of a smooth and irreducible hypersurface $L_b$. Therefore we may write $f^* b = n_b L_b$ with $n_b \ge 1$. When $n_b >1$ then the fiber of $f$ over $b$ is a multiple fiber with multiplicity $n_b$.

\subsection{Orbifold  base}
It is convenient to associate to any quasi-smooth fibration the $\mathbb Q$-divisor
\[
    \Delta = \sum_{b \in B} \frac{n_b -  1}{n_b} \cdot b  \in \Div(B) \otimes \mathbb Q \, ,
\]
and consider the pair $(B,\Delta)$ as the orbifold base of the fibration in the sense of \cite[Section 1]{zbMATH02123576}.

Notice that $f^*\Delta \in \Div(X)$ and satisfies
\[
    f^* \Delta = \sum_{b \in B} (n_b -1) L_b  \, .
\]
A standard computation, which traces back to Darboux, shows that $N^*_{\F} = f^* \omega_B \otimes \mathcal O_X(f^* \Delta)$ (see \cite[Section 2.3, Example 5]{zbMATH02150908}).

An orbifold base $(B,\Delta)$ is called good (also called uniformizable) if there exists a curve $C$ and a discrete subgroup $G\subset \Aut(C)$ such that
$B \simeq C/G$ and the ramification divisor of the quotient morphism $\pi: C \to C/G \simeq B$ equals $\pi^*\Delta$. We will call the triple $(C,G,\pi)$ an unfolding (also called uniformization) of $(B,\Delta)$. Otherwise, the orbifold base $(B,\Delta)$ is called bad.

Most orbifolds are good as established by Bundgaard-Nielsen-Fox theorem stated below. For a brief account of its history, the reader can consult  \cite[Theorem 1.4]{uludag-07-MR2306159}. We also refer to \cite{zbMATH04046721} and \cite[Chapter 2]{zbMATH00558225} for more information about orbifolds.

\begin{thm}\label{T:BNF}
    If $B$ is a compact Riemann surface and $(B,\Delta)$ is a bad orbifold then $B \simeq \mathbb P^1$. Moreover, $\Delta = \frac{n-1}{n} p$, for $n \ge 2$ and $p \in \mathbb P^1$,  or  $\Delta=\frac{n-1}{n} p + \frac{m-1}{m} q$ for distinct integers $n, m\ge2$ and distinct points $p,q \in \mathbb P^1$.
\end{thm}

\subsection{Orbifold base of a transversely projective fibration is good}

\begin{prop}\label{P:fibrations vs structures}
    Let $f: X \to B$ be a quasi-smooth fibration and let $\F$ be the foliation defined by it. Then $\F$ admits a smooth transversely projective structure if, and only if, the orbifold base $(B,\Delta)$ is good.
\end{prop}
\begin{proof}
    Assume first that the orbifold base of $f$ is good. Let $(C, G, \pi)$ be an unfolding of $(B,\Delta)$. The normalization of the fiber
    product $\Tilde X = X \times_B C$ comes endowed with a smooth fibration $\tilde f: \Tilde X \to C$. The canonical projective structure on
    $C$ induces a transversely projective structure on the foliation defined by $\tilde f$. Moreover, this transversely projective structure is invariant by the natural action of $G$ on $\Tilde X$. Therefore, it descends to a smooth projective structure for $\F$.

    Assume now that $\F$ admits a smooth transversely projective structure. Let $\rho : \pi_1(X) \to \Aut(\mathbb P^1)$ be the monodromy representation of the structure. Selberg's lemma implies the existence of a finite index normal subgroup $\Gamma \subset \pi_1(X)$ such that $\rho(\Gamma)$ is torsion free.
    Let $G = \pi_1(X)/\Gamma$ and consider the Galois covering $\pi: \Tilde X \to X$ with group $G$ determined by $\Gamma$. Since the holonomy representation of leaves factor through the monodromy representation of the projective structure, the foliation $\pi^* \F$ is defined by a smooth fibration $\tilde f : \Tilde X \to C$, for some curve $C$. Moreover, the group $G$ acts on $C$ in such a way that $C/G \simeq B$ and, $(C,G, C \to B)$ is an unfolding of $(B,\Delta)$. It follows that the orbifold base is good.
\end{proof}

\begin{example}[Quasi-smooth fibrations which are not transversely projective foliations] \label{Ex: Hopf}
    Let $\lambda \in \mathbb C^*$ be a complex number with modulus strictly smaller than one. Let
    $\mathbb Z$ act on $\tilde X = \C^2 - \{ (0,0) \}$ through $\gamma(x,y) = (\lambda^n x, \lambda^m y)$ where  $n, m$ are distinct positive integers.
    Let  $X  = \Tilde X / \mathbb Z$ be the associated Hopf surface. If $\gcd(n,m)=1$ then the meromorphic $f(x,y) = ( x^m : y^n)$ from $\tilde X$ to $\mathbb P^1$ defines a quasi-smooth fibration on $X$ with orbifold base
    \[
        \left(\mathbb P^1, \frac{m -1}{m} p + \frac{n-1}{n} q \right)
    \]
    where $ p =  (0:1)$ and $q = (1:0)$. Since this orbifold base is bad, the foliation on $X$ is not transversely projective.
\end{example}

\subsection{Quasi-smooth fibrations on compact K\"ahler surfaces}

\begin{thm}\label{T: fibration in Kahler is transversely projective}
    Let $X$ be a compact K\"ahler surface. If $f : X \to B$ is a quasi-smooth fibration then the orbifold base of $f$ is good.
\end{thm}

Combining Proposition \ref{P:fibrations vs structures} with Theorem \ref{T: fibration in Kahler is transversely projective} we obtain the following statement which we record for later use.

\begin{cor}\label{C: fibration in Kahler is transversely projective}
    Let $X$ be a compact K\"ahler surface. If $f : X \to B$ is a quasi-smooth fibration then, after passing to an  étale covering of $X$, the fibration $f$ becomes a smooth fibration.
\end{cor}

The proof of Theorem \ref{T: fibration in Kahler is transversely projective} relies on a topological observation which holds for
compact complex manifolds of arbitrary dimension.

\begin{lemma} \label{L: Euler zero}
    Let $X$ be a compact complex manifold. If $f : X \to B$ is a quasi-smooth fibration with bad orbifold base then
    the topological Euler characteristic of the fibers and of $X$ is zero.
\end{lemma}
\begin{proof}
    Theorem \ref{T:BNF} implies that $B$ is equal to $\mathbb P^1$
    and $\Delta = \frac{n-1}{n} p$ or $\Delta = \frac{n-1}{n} p  + \frac{m-1}{m} q$ with $n \neq m$. Moreover, in the second case,
    after replacing $X$ by an étale covering of order $\gcd(n,m)$, we may assume that $\gcd(n,m)=1$.  As before, let us denote
    by $L_p$, $L_q$ the supports of the multiple fibers. The open subset of regular values of $f$ will be denoted by $B^{\circ}$ and
    its pre-image under $f$ by $X^{\circ}$.

    The support of a multiple fiber, say $L_p$, seen as leaf of the induced foliation, has finite holonomy and, by Reeb Local Stability Theorem,
    the smooth fibers of $f$ (which are all diffeomorphic to a general fiber $L$) are (topological) $n$-sheeted coverings of $L_p$. Therefore,
    \begin{equation} \label{E: euler characteristic}
        \chi_{top}(L) = n \chi_{top}(L_p) = m \chi_{top}(L_q).
    \end{equation}

    Consider the local systems over $B^0$ given by $R^\ell f_* \C$. If $B^{\circ} \simeq \C$ then the simply-connectedness of $\C$ implies that these local systems have trivial monodromy. The same holds true when $B^{\circ} \simeq \C^*$ but now because the order of the monodromy around $p$ divides $n$ and around $q$ divides $m$.

    Let $g: L \rightarrow L$ be a non-trivial deck transformation.  By Lefschetz trace formula  (see for instance \cite[Theorem 2C.3]{zbMATH02103273}),
    \[
         \sum_{i=0}^{2 \dim L} (-1)^i \mathrm{tr} \Big( g_*: H_i(L, \C) \rightarrow H_i(L, \C) \Big) = 0 .
    \]
    Observe that $g_*$ corresponds to the monodromy representation of a path $\gamma$ around $p$. Therefore, since the monodromy of the local systems $R^\ell f_* \C$ is trivial,  $g_* = \mathrm{Id}$. Hence,
    \[
    \sum_{i=0}^{2 \dim L} (-1)^i \mathrm{tr}\Big(g_*: H_i(L, \C) \rightarrow H_i(L, \C)\Big) = \sum_{i=0}^{2 \dim L} (-1)^i h_i(L, \C) = \chi_{top}(L)
    \]
    Hence, $\chi_{top}(L)=0$. Also by \ref{E: euler characteristic}, $\chi_{top}(L_p)= \chi_{top}(L_q)=0$. This shows that every fiber has
    topological Euler characteristic equal to zero, as claimed.

    To verify the  vanishing of the topological Euler characteristic of $X$, first note that  the additivity of the topological Euler characteristic implies
    \[
        \chi_{top}(X)=\chi_{top}(X^{\circ}) + \chi_{top}(L_p) +  \chi_{top}(L_q) \, .
    \]
    The vanishing of $\chi_{top}(X^{\circ})$ follows from the multiplicativity of the Euler characteristic for smooth fibrations. The vanishing of $\chi_{top}(X)$ follows.
\end{proof}

\subsection{Proof of Theorem \ref{T: fibration in Kahler is transversely projective}}
    Aiming at a contradiction, let us assume that the orbifold base of $f$ is bad. We will keep the notation of the proof of Lemma \ref{L: Euler zero}. Moreover, in the case $f$ has two multiple fibers we can assume, without loss of generalities, that the multiplicities of the fibers are relatively prime.

    Let $\iota: L \rightarrow X$ be the inclusion of the general fiber $L$ of $f$.
    According to  Deligne's invariant subspace theorem, see for instance \cite[Theorem 16.2]{zbMATH01927232},
    the image of $\iota^*:  H^1(X, \C) \rightarrow H^1(L, \C)$ is the subspace invariant by the monodromy representation of the local system $R^1f_* \C$ defined on $B^\circ$. Because it is trivial (see proof of Lemma \ref{L: Euler zero}), we conclude that $\iota^*$ is surjective. Furthermore, $\iota^*$ defines a morphism of Hodge structures. Hence we get a surjective morphism
    \[
        \iota^*: H^0(X,\Omega^1_X) \rightarrow H^0(L, \Omega^1_{L}).
    \]
    Because $L$ is of genus one, we have a nowhere vanishing $1$-form $\omega_L \in H^0(L, \Omega^1_L)$. 
    
    Let $\omega \in H^0(X, \Omega^1_X)$ such that $\iota^* \omega = \omega_L$ and consider the cohomology class of $\alpha = [\omega \wedge \overline \omega] \in H^2(X, \mathbb C)$. Note that the contraction of $\alpha$ with the homology class of $L$ is equal to $\int_L \alpha$ and, therefore,  non-zero. It follows that the same holds true for every other fiber of $f$ (multiple or not) since their homology classes all lie on the same line. Hence the $1$-form $\omega$ is non-zero in every fiber and it  defines a foliation $\G$ completely transversal to $\F$. In particular, $T_X = T_{\F} \oplus T_{\G}$.  Using \cite[Theorem C]{zbMATH01440936}, see also \cite{zbMATH05145488} and \cite[Theorem 1.4]{zbMATH07634625}, the universal cover of $\pi: \til{X} \rightarrow X$ is the product of two simply-connected Riemann surfaces, and the foliation $\pi^*(\F)$ is the fibration over one of them. This would imply that $B$ is a good orbifold, a contradiction. 
\qed
\subsection{Quasi-smooth fibrations on surfaces with bad orbifold base}

From the Theorem \ref{T: fibration in Kahler is transversely projective}, we know that if $f:X \rightarrow \Pj^1$ is a quasi-smooth fibration  with bad orbifold base, then $X$ must be a non K\"ahler compact surface. Moreover, in Example \ref{Ex: Hopf}, we constructed examples of quasi-smooth fibrations on Hopf surfaces with bad orbifold base. From this, one can ask if there are other examples. The next proposition shows that this is the only case. 

\begin{prop}
    Let $X$ be a compact complex surface and $f:X \rightarrow \Pj^1$ be a quasi-smooth fibration with bad orbifold base. Then $X$ is a Hopf surface.
\end{prop}
\begin{proof}
    Let $f:X \rightarrow \Pj^1$ be a fibration with a bad orbifold base $(\Pj^1, \frac{m-1}{m}p + \frac{n-1}{n}q)$ (for the case $(\Pj^1, \frac{m-1}{m}p)$ the computation is similar). By Lemma \ref{L: Euler zero}, $f$ is an elliptic fibration, and since the fibers are irreducible curves, they can not be $(-1)$-curves. That is, $f$ is a relatively minimal elliptic fibration.
    
    The first step is to conclude that $\kod(X)= -\infty$. By \cite[Chapter V, Proposition 12.5]{zbMATH02008523},
    $\kod(X)= -\infty$ if, and only if, $\delta(f)<0$, where $\delta(f)$ of a elliptic fibration $f:X \rightarrow B$ is defined as 
    \[
    \delta(f):= \chi(X,\mathcal{O}_X) + \left(2g(S) - 2 + \deg(\Delta) \right).
    \]
In our situation, we have that $\chi(X,\germe_X) = deg(f_*\omega_{X|\Pj^1})$ by \cite[Chapter III, Theorem 18.2]{zbMATH02008523}, and by \cite[Chapter III, Theorem 18.2]{zbMATH02008523}, we have $deg(f_*\omega_{X|\Pj^1}) = 0$. Therefore,

\[
\delta(f) = \chi(X,\germe_X) - \frac{1}{m} - \frac{1}{n} = - \frac{1}{m} - \frac{1}{n} < 0,
\]
and we conclude that $\kod(X) =- \infty$.

From Theorem \ref{T: fibration in Kahler is transversely projective} and Lemma \ref{L: Euler zero}, we see that $X$ must be a non K\"ahler surface. By \cite[Chapter V, Theorem 18.6]{zbMATH02008523}, it follows that $X$ is a Hopf surface.

\end{proof}

\subsection{Quasi-smooth fibrations on projective threefolds} It seems natural to wonder whether Theorem \ref{T: fibration in Kahler is transversely projective} holds for quasi-smooth fibrations over curves on arbitrary compact K\"ahler manifolds. We believe that this is the case, and we could check it for quasi-smooth fibrations on projective threefolds.

\begin{thm}\label{T:fibration on 3folds}
    Let $X$ be a projective manifold of dimension three. If $f : X \to B$ is a quasi-smooth fibration then the orbifold base of $f$ is good.
\end{thm}
\begin{proof}
    After applying the minimal model program (MMP) for codimension one smooth foliations on projective threefolds \cite[Theorem 1.3]{zbMATH07132579} we can assume that every fiber of $f$ is minimal, \ie does not contain a smooth curve isomorphic to $\mathbb P^1$ with self-intersection $-1$.

    As in the proof of Theorem \ref{T: fibration in Kahler is transversely projective}, we will assume that the orbifold base is bad and look for a contradiction. 
    
    Lemma \ref{L: Euler zero} implies that every fiber of $f$ has zero topological Euler characteristic. An inspection of the Enriques-Kodaira classification \cite[Chapter VI, Table 10]{zbMATH02008523} shows that the fibers of $f$ are
    \begin{enumerate}
        \item $\mathbb P^1$-bundles over elliptic curves when they have negative Kodaira dimension; or
        \item bi-elliptic surfaces or abelian surfaces when they have Kodaira dimension zero; or
        \item minimal properly elliptic surfaces when they have Kodaira dimension one.
    \end{enumerate}

    If the fibers have negative Kodaira dimension then $X$ itself is a $\mathbb P^1$-bundle over a smooth projective surface $S$ and the fibration
    descends to a fibration on $S$. In this case, the result follows from Theorem \ref{T: fibration in Kahler is transversely projective}.

    If the fibers are abelian surfaces then we can argue as in the proof of Theorem \ref{T: fibration in Kahler is transversely projective}, but now using $2$-forms instead of $1$-forms, to obtain a one-dimensional foliation $\G$ everywhere transverse to $\F$.  It follows from \cite{zbMATH05145488} that the universal covering is the product $\mathbb C^2 \times B = \mathbb C^2 \times \mathbb P^1$ with $f$ lifting to the fibration defined by projection to $\mathbb P^1$. It follows that $f$ is good.

    If the fibers are bi-elliptic surfaces or have Kodaira dimension one then there exists a smooth foliation $\mathcal E$ by elliptic curves tangent to the fibration. Since $\mathcal E$ is smooth, it must be an isotrivial fibration. Therefore we can apply \cite[Theorem 1.4]{zbMATH07634625} to produce a finite étale covering $p:Y \to X$ such that $p^*\mathcal E$ is a locally trivial elliptic fibration (no multiple fibers). Again, the fibration $f$ descends to a fibration on a smooth projective surface and we can conclude using \ref{T: fibration in Kahler is transversely projective} as in the case of negative Kodaira dimension .
\end{proof}

\begin{remark}
    We only used the projectiveness of $X$ in order to apply the results on the MMP for codimension one foliations. It is very likely that the argument above can be adapted to show the result for compact K\"ahler threefolds.
\end{remark}

\section{Smooth Turbulent foliations on compact K\"ahler surfaces}\label{S:turbulent}
In this section, we are going to study when a smooth turbulent foliation on a surface admits a transversely projective structure.

\subsection{Definition}
Let $X$ be a complex surface. A smooth foliation $\F$ on $X$ is called turbulent if there exists a proper  fibration $f:X\rightarrow B$ such that the fibers of $f$ are curves of genus one, and $\F$ is completely transverse to the general fiber of $f$. We will say that $\F$ is turbulent with respect to $f$, and that $f$ is the reference fibration of $\F$.

It follows from the definition, that the reference fibration of a smooth turbulent foliation is automatically quasi-smooth. Indeed, as explained in \cite[Section 4.3]{zbMATH02150908}, if the support of a fiber of $f$ is not smooth then it is not irreducible and one of its components is invariant by $\F$. Moreover, since such component have negative self-intersection, Camacho-Sad formula implies that it must contain a singularity of the foliation. But this contradicts the smoothness of $\F$.

\subsection{Turbulent foliations without invariant fibers}

\begin{prop}\label{P: turbulent foliation without F-invariant fiber is transversely projective}
    Let $\F$ be a turbulent foliation on a compact K\"ahler surface $X$ with reference fibration $f:X \to B$. If no fiber of $f$ is $\F$-invariant then
    $\F$ is transversely projective.
\end{prop}
\begin{proof}
    If the fibration is smooth than the result follows from Example \ref{E:transverse to fibrations}. If the fibration is only quasi-smooth then the natural orbifold structure on the base is good according to Theorem \ref{T: fibration in Kahler is transversely projective}. Therefore, there exists an étale Galois covering $\pi : Y \to X$  induced by a ramified covering $C \to B$ such that the reference fibration of  the turbulent foliation $\pi^* \F$ is smooth. Example \ref{E:transverse to fibrations} provides a transverse projective structure for $\pi^* \F$. Moreover, since such projective structure is obtained by extending the canonical projective structure on a fiber of $f$, it must be invariant by the action of the deck transformations of $\pi$. Hence it descends to a transversely projective structure for $\F$.
\end{proof}

\subsection{Multiple and invariant fibers}
Let $X$ be a smooth complex surface, let $\F$ be a smooth turbulent foliation on $X$, and let $f: X \to B$ be its reference fibration.
Let $B^{\circ} \subset B$ be the set of regular values of $f$, and let $B^{\ast}\subset B^{\circ}$ be equal to $B^{\circ}$ minus the points
corresponding to smooth $\F$-invariant fibers.

The transversality of $\F$ and the general fiber of $f$ implies that the restriction of $f$ to $\pi^{-1}(B^{\circ})$ is locally trivial.
Moreover, over discs $\mathbb D \subset B^{\circ}$ we have that $f^{-1}(\mathbb D) \simeq \mathbb D \times E$ and the foliation $\F$ is either conjugated to
the horizontal foliation defined by $dz$, when none of the fibers are  $\F$-invariant, or, there exists a meromorphic $1$-form $\alpha$ on $\mathbb D$ with polar divisor equal to the union of the points below the $\F$-invariant fibers such that $\restr{\F}{f^{-1}(\mathbb D)}$ is defined by $\alpha + dz$.

The description of $f$ and $\F$ around a multiple fiber of multiplicity $n$ is similar. If $\mathbb D \subset B$ is a disc with $\mathbb D^*\subset B^{\circ}$ then $f^{-1}(\mathbb D)$ is the quotient of $\mathbb D \times E$ by the automorphism
\[
    (x,z) \mapsto \left(\exp(2 \pi \sqrt{-1} / n) x , z + \tau \right) ,
\]
where $\tau \in E$ is a torsion point of order $n$, according to \cite[Section 13 of Chapter V]{zbMATH02008523}. The foliation $\F$ is like in the smooth case except for the extra condition, in the case $f^{-1}(0)$ is a $\F$-invariant fiber, that the $1$-form $\alpha$ is invariant by the autmorphism $x \mapsto \exp(2 \pi \sqrt{-1} / n) x$.

\begin{lemma}\label{L: turbulent foliation given by closed meromorphic 1-form}
    Let $\F$ be a smooth turbulent foliation on a complex surface $X$. Then, after passing to a finite étale covering of $X$, the foliation is given by a closed meromorphic 1-form without zeroes.
\end{lemma}
\begin{proof}
    The discussion carried out above implies the existence of an open cover $\mathcal U = \{ U_i\}$ of $B$ such that
    $\restr{\F}{f^{-1}(U_i)}$ is defined by a closed meromorphic $1$-form $\omega_i$.  We can further assume that the integral of $\omega_i$
    along generators of the homology group of the fibers does not vary with the base point. Moreover, we can also assume that this integral is equal to one at one of the generators. This rigidifies the choice of the $\omega_i$ in such a way
    that over non-empty intersections $U_i \cap U_j$, the $1$-forms $\omega_i$ and $\omega_j$ differ by multiplication of a locally constant
    function with values in $\{ \pm 1\}$, or $\{ \pm 1 , \pm i\}$, or $\{ \exp(2 \pi \sqrt{-1} j /6 ; j = 0 \ldots 5 \}$ depending on the automorphism group of the general fiber. In any case, we obtain a representation  of $\pi_1(B)$ into a finite cyclic group. After doing a base change of $f: X \to B$ to the associated étale covering of $B$, it is clearly possible to define the transformed foliation by a closed meromorphic $1$-form.
\end{proof}

We will need the following fact concerning the non existence of first integral on neighborhoods of $\F$-invariant fibers of a turbulent foliation.

\begin{lemma}\label{L: non existence of first integral for turbulent foliation}
    Let $\F$ be a turbulent foliation on a smooth complex surface $X$. Suppose that $\F$ leaves invariant some fiber of the reference fibration.  Then $\F$ does not admit a non-constant meromorphic first integral.
\end{lemma}
\begin{proof}
    Let $f : X \to B$ be the reference fibration of $\F$ and let $g$ be a meromorphic first integral of $\F$.
    First observe that $g$ does not have indeterminacy points. Indeed, an indeterminacy point of $g$ would correspond a point in the intersection of several distinct leaves of $\F$. The smoothness of $\F$ prevents the existence of such point. Therefore, $g$ is not only a meromorphic map, but it is actually a morphism. Since it is constant along the leaves of $\F$, it must map the $\F$-invariant $F_0$ fiber to a point. Hence all the fibers sufficiently
    close to $F_0$ must also be mapped to a point. Since $\F$ is turbulent, this implies that $g$ is constant.
\end{proof}

\subsection{Transversely projective structures}

\begin{thm}\label{T: turbulent is transversely projective if and only if (A) or (M)}
    Let $\F$ be a smooth turbulent foliation on a compact K\"ahler surface $X$. Assume that $\F$ leaves invariant  at least one fiber of the reference fibration.  Then, $\F$ admits a transversely projective structure if, and only if, after a finite étale covering, $\F$ is given by a closed meromorphic 1-form $\omega$ of one of the following types
    \begin{enumerate}
        \item\label{I:A} $\omega$ is without residues and $(\omega)_{\infty} = 2 \sum C_i$, where $C_i$ are irreducible smooth curves, $C_i \cap C_j =\emptyset, \forall i\neq j$; or 
        \item\label{I:M} $\omega$ is logarithmic with residues of the form $\pm \lambda$, for some $\lambda \in \C^*$.
    \end{enumerate}
\end{thm}
\begin{proof}
    Assume first that $\F$ is transversely projective. After passing to a finite étale cover we can assume that the reference fibration has no multiple fibers (Corollary \ref{C: fibration in Kahler is transversely projective}) and that $\F$ is defined by a closed meromorphic $1$-form  $\omega$ without zeros (Lemma  \ref{L: turbulent foliation given by closed meromorphic 1-form}).  Let $C_1,\ldots, C_k$ the $\F$-invariant curves that appear in $(\omega)_{\infty}$. Each $C_i$ corresponds to a fiber $f^{-1}(p_i)$. Let $U \subset B$ be a simply connected open set of $B$, such that $p_i \in U$ for every $U$. Let $X_U = f^{-1}(U)$. By Ehresmann's Theorem,  $X_U$  is $C^{\infty}$-diffeomorphic to the product of an arbitrary fiber and $U$. In particular,  $\pi_1(X_U)$ is abelian. The  foliation $\restr{\F}{X_U}$  is still turbulent and transversely projective, and we are going to prove that $\restr{\omega}{X_U}$ is of type (\ref{I:A}) or (\ref{I:M}).

    Let $\mathcal{P} = (\pi:Y \rightarrow X_U, \G, \alpha: X_U \rightarrow Y)$ be a projective structure for $\restr{\F}{X_U}$, and $\rho: \pi_1(X_U) \rightarrow \mathrm{PSL}(2, \mathbb{C})$ be its  monodromy representation. Because $\pi_1(X_U)$ is abelian, the representation $\rho$, up to conjugation,  fits into one the following descriptions.
    \begin{enumerate}[label=(\roman*)]
        \item\label{I: multiplicative monodromy} The image of $\rho$ is contained in the subgroup  $(\mathbb{C}^*,\cdot) \subset \Aut(\C) \subset \Aut(\mathbb P^1)$. In this case, the foliation $\G$ is defined by a logarithmic 1-form with residues $\pm 1$ locally of the form  $\frac{dz}{z}$, where $z$ is a coordinate on $\mathbb P^1$.
        \item\label{I: additive monodromy}  The image of $\rho$ is contained $(\mathbb{C}, + ) \subset \Aut(\C) \subset \Aut(\mathbb P^1)$. In this case, $\G$ is given by a meromorphic 1-form without residues with is locally of the form $dz$.
        \item\label{I: involutions}  The image of $\rho$ has image contained in the subgroup of order four $\left\{\mathrm{id}, z\mapsto -z, z\mapsto \frac{1}{z}, z\mapsto -\frac{1}{z} \right\} \subset \Aut(\mathbb P^1)$. In this case, $\G$ has admits a global meromorphic first integral locally of the form $f(z) = (z - \frac{1}{z})^2$.
    \end{enumerate}
    In case \ref{I: involutions}, $\restr{\F}{X_U} = \alpha^*(\G)$ has meromorphic first integral contradicting Lemma \ref{L: non existence of first integral for turbulent foliation}.  For the cases  \ref{I: multiplicative monodromy} and \ref{I: additive monodromy}, let $\eta$ be the meromorphic 1-form defining $\G$. Because $\alpha^*(\eta)$ defines $\restr{\F}{X_U}$, there exists a meromorphic function $F \in \mathbb C(X_U)$  such that
    \[
        \alpha^*(\eta) = F \cdot \restr{\omega}{X_U}.
    \]
    Because $\eta$ and $\omega$ are both closed, then $F$ is first integral for $\restr{\F}{X_U}$. Therefore, $F$ is constant by Lemma \ref{L: non existence of first integral for turbulent foliation}. The transversality between $\alpha$ and $\G$ implies that the order of poles and residues of $\restr{\omega}{X_U}$ are the same as the ones of $\eta$. The result follows.

    Suppose now that there exists a finite étale covering $\pi : \Tilde{X} \to X$ such that $\Tilde \F = \pi^* \F$ is defined by a closed meromorphic 1-form $\omega$ of type (\ref{I:A}) or (\ref{I:M}). Without loss of generality, we can assume that $\pi$ is Galois with Galois group $G$.  In case (\ref{I:A}), local primitives $\int \omega$ form a collection of distinguished first integrals for $\Tilde \F$ showing that $\Tilde \F$ is transversely projective. Similarly, in case (\ref{I:M}), we consider the local first integrals defined by $\exp\left( \lambda^{-1} \int \omega \right)$ to produce a transversely projective structure for $\Tilde \F$.

    Since $\Tilde \F$ is defined by a closed meromorphic $1$-form without zeros and with non-empty polar divisor,  $N_{\Tilde \F} = \germe_{\Tilde{X}}((\omega)_{\infty})$ is effective and non-zero. Therefore, $H^0((N^*_{\Tilde \F})^{\otimes 2})=0$, and thus by Proposition \ref{P:estruturas}, the transversely projective structure of $\Tilde \F$ is unique. In particular, it is invariant by the Galois group $G$. Thus, the transversely projective structure descends to a transversely projective structure for $\F$. The theorem follows.
\end{proof}

The next example illustrates how the theorem above, applied in some special cases, provides examples of turbulent foliations that admit (and does not admit) transversely projective structures. 

\begin{example}
    Let $E$ be an elliptic curve, and let $X = E \times \Pj^1$ with projections $\pi_1: X \rightarrow E$ and $\pi_2: X \rightarrow \Pj^1$. The smooth turbulent foliations on $X$ are given by closed meromorphic 1-forms of the form $\pi_1^*(\omega_1)+\pi_2^*(\omega_2)$, where $\omega_1$ is a holomorphic non-vanishing 1-form on $E$, and $\omega_2$ is a meromorphic 1-form on $\Pj^1$ (see \cite[Theorem 1.2]{Ghys96Feuilletages}). Then, $\mathcal{F}$ is smooth transversely projective if and only if $\omega_2$ satisfies one the conditions of Theorem \ref{T: turbulent is transversely projective if and only if (A) or (M)}. For instance, considering a coordinate $z$ for $\Pj^1$, both $\omega_2 = dz$ and $\omega_2 = dz/z$ provide examples of transversely projective foliations, while $\omega_2 = dz/z^3$ provides an example of a turbulent foliation without transversely projective structure. More generally, one can verify that a generic smooth turbulent foliation in $X$ does not admit a transversely projective structure.
\end{example}

\section{Foliations on Hopf and Inoue surfaces}\label{S:Hopf}

\subsection{Foliations on Inoue surfaces}
An Inoue surface $X$ is a quotient of $\mathbb H \times \C$ by a (non-abelian) subgroup of $\Aut(\mathbb H \times \mathbb C) \subset \Aut(\mathbb C^2)$  generated by one diagonal linear map and three translations, see \cite[Chapter V, Section 19]{zbMATH02008523}. The foliations defined by the natural projections to $\mathbb H$ and $\mathbb C$ descend to foliations on $X$ and are the unique foliations on $X$. Both are clearly transversely projective (actually, they are transversely affine).

\subsection{Foliations on Hopf surfaces} \label{SubS: Linear Hopf}
A Hopf surface $X$ is a compact complex surface which has as universal covering $W = \mathbb C^2 -\{(0,0)\}$, see \cite[Chapter V, Section 18]{zbMATH02008523}. If $\pi_1(X)$ is isomorphic to $\mathbb Z$ then $X$ is called a primary Hopf surface, otherwise $X$ is a secondary Hopf surface. Any secondary Hopf surface admits a finite étale covering which is a primary Hopf surface. Primary Hopf surfaces in their turn, are $C^{\infty}$-diffeomorphic to $S^3 \times S^1$ and biholomorphic to the cyclic quotient of $W$ by the group generated by 
either $\varphi(x,y) = (\alpha x, \beta y)$  where $\alpha, \beta$ are complex numbers satisfying $0 < |\alpha| \le |\beta| < 1$ or 
 $\varphi(x,y) \mapsto (\alpha^n x + \tau \cdot y^n, \alpha y)$, where $\alpha$ is as before, $\tau \in \mathbb C^*$,  and $n \ge 2$ is an integer. 

\subsection{Linear primary Hopf surfaces} Assume that we are in the first case, that is $\varphi :W \to W$ is a linear contraction. If the multiplicative semi-group of $\mathbb C^*$ generated by $\alpha$ and $\beta$ has rank two then the corresponding primary Hopf surface is of algebraic dimension zero and the foliations on it are quotients of the foliations defined by the $1$-forms $\lambda \frac{dx}{x} + \mu \frac{dy}{y}$ where $\lambda$ and $\mu$ are complex numbers, not both zero. If instead the multiplicative semi-group of $\mathbb C^*$ generated by $\alpha$ and $\beta$ has rank one, i.e. there exists relatively prime positive integers $n,m >0$ such that $\alpha^n = \beta^m$ then the algebraic dimension of $X$ is equal to one. Furthermore, the meromorphic function $x^n/y^m$ descends to a meromorphic function $f$ on $X$ and is such that $\mathbb C(X) = \mathbb C(f)$. Notice that $f$ has no indeterminacy points and defines a morphism from $X \to \mathbb P^1$ with elliptic fibers. Note that the orbifold structure on $\mathbb P^1$ is good only when $n=m=1$. In any case, every other foliation $\F$ on the primary Hopf surface $X$ is turbulent with respect to this fibration. Moreover, they are all defined by a closed meromorphic $1$-form $\omega$. We point out that $\F$ has at least one invariant fiber and hence, by Lemma \ref{L: non existence of first integral for turbulent foliation}, it does not have a  meromorphic first integrals. Alternatively, this also follows from the fact the algebraic dimension of $X$ is equal to one. As a consequence, we obtain that the $1$-form $\omega$ is, up to multiplication by a non-zero constant, the unique closed meromorphic $1$-form defining $\F$.

\subsection{Non-linear primary Hopf surfaces} We now discuss the second possibility for $\varphi$, that is $\varphi$ is non-linear and of the form $\varphi(x,y) \mapsto (\alpha^n x + \tau \cdot y^n, \alpha y)$. In  this case, the primary Hopf surface $X$ has algebraic dimension zero and  the foliations on $X$ are obtained as quotients of the foliations on $W$ defined by the meromorphic $1$-forms
\[
     \lambda \frac{dy}{y}   + \mu d \left( \frac{x}{y^n} \right) \, ,
\]
for $\lambda, \mu \in \mathbb C$. Notice that they are all closed and invariant by $\varphi$, thus they descend to closed meromorphic $1$-forms
on $X$. 

\subsection{Synthesis}
Having the description of the foliations on primary Hopf surfaces at hand, we can argue as in the proof of Theorem \ref{T: turbulent is transversely projective if and only if (A) or (M)} to obtain the following result.

\begin{prop}\label{P: foliations on hopf surface with transversely projective structure}
    Let $\F$ be a foliation on a Hopf surface. Assume that the general leaf of $\F$ is not compact.  Then $\F$ is transversely projective if, and only if, after a finite étale covering, $\F$ is given by a closed meromorphic 1-form $\omega$ of one of the following types
    \begin{enumerate}
        \item $\omega$ is without residues and $(\omega)_{\infty} = 2 \sum C_i$, where $C_i$ are irreducible smooth curves, $C_i \cap C_j =\emptyset, \forall i\neq j$; or
        \item $\omega$ is logarithmic with residues of the form $\pm \lambda$, for some $\lambda \in \C^*$.
    \end{enumerate}
\end{prop}

We provide an application of the preceding proposition, exhibiting cases of turbulent foliations that either possess or lack a transversely projective structure, depending on the choice of the defining 1-form.

\begin{example}
    Considering the description of the non-linear primary Hopf surfaces, Proposition \ref{P: foliations on hopf surface with transversely projective structure} implies that a foliation $\F$ is transversely projective if, and only if, $\lambda=0$ and $n=1$.
\end{example}

\providecommand{\bysame}{\leavevmode\hbox to3em{\hrulefill}\thinspace}
\providecommand{\MR}{\relax\ifhmode\unskip\space\fi MR }
\providecommand{\MRhref}[2]{%
  \href{http://www.ams.org/mathscinet-getitem?mr=#1}{#2}
}
\providecommand{\href}[2]{#2}


\begin{thebibliography}{10}

\bibitem{zbMATH02008523}
Wolf~P. Barth, Klaus Hulek, Chris A.~M. Peters, and Antonius Van~de Ven, \emph{Compact complex surfaces}, 2nd enlarged ed., Ergeb. Math. Grenzgeb., 3. Folge, vol.~4, Berlin: Springer, 2004.

\bibitem{zbMATH01440936}
Arnaud Beauville, \emph{Complex manifolds with split tangent bundle}, Complex analysis and algebraic geometry. A volume in memory of Michael Schneider, Berlin: Walter de Gruyter, 2000, pp.~61--70.

\bibitem{biswas2024holomorphic}
Indranil Biswas and Sorin Dumitrescu, \emph{Holomorphic foliations with no transversely projective structure}, 2024.

\bibitem{zbMATH01088813}
Marco Brunella, \emph{Feuilletages holomorphes sur les surfaces complexes compactes}, Ann. Sci. {\'E}c. Norm. Sup{\'e}r. (4) \textbf{30} (1997), no.~5, 569--594.

\bibitem{zbMATH02150908}
\bysame, \emph{Birational geometry of foliations}, Publ. Mat. IMPA, Rio de Janeiro: Instituto Nacional de Matem{\'a}tica Pura e Aplicada (IMPA), 2004.

\bibitem{zbMATH05145488}
Marco Brunella, Jorge~Vit{\'o}rio Pereira, and Fr{\'e}d{\'e}ric Touzet, \emph{K{\"a}hler manifolds with split tangent bundle}, Bull. Soc. Math. Fr. \textbf{134} (2006), no.~2, 241--252.

\bibitem{zbMATH03907374}
C{\'e}sar Camacho and Alcides Lins~Neto, \emph{Geometric theory of foliations. {Transl}. from the {Portuguese} by {Sue} {E}. {Goodman}}, Boston - {Basel} - {Stuttgart}: {Birkh{\"a}user}. {V}, 205 p. {DM} 98.00 (1985)., 1985.

\bibitem{zbMATH02123576}
Fr{\'e}d{\'e}ric Campana, \emph{Orbifolds, special varieties and classification theory.}, Ann. Inst. Fourier \textbf{54} (2004), no.~3, 499--630.

\bibitem{zbMATH07634625}
St{\'e}phane Druel, Jorge~Vit{\'o}rio Pereira, Brent Pym, and Fr{\'e}d{\'e}ric Touzet, \emph{A global {Weinstein} splitting theorem for holomorphic {Poisson} manifolds}, Geom. Topol. \textbf{26} (2022), no.~6, 2831--2853.

\bibitem{zbMATH00558225}
Robert Friedman and John~W. Morgan, \emph{Smooth four-manifolds and complex surfaces}, Ergeb. Math. Grenzgeb., 3. Folge, vol.~27, Berlin: Springer-Verlag, 1994.

\bibitem{Ghys96Feuilletages}
{\'E}tienne Ghys, \emph{Holomorphic foliations of codimension 1 on complex homogeneous spaces}, Ann. Fac. Sci. Toulouse, Math. (6) \textbf{5} (1996), no.~3, 493--519 (French).

\bibitem{zbMATH00050350}
Claude Godbillon, \emph{Feuilletages: {\'e}tudes g{\'e}om{\'e}triques}, Prog. Math., vol.~98, Basel etc.: Birkh{\"a}user Verlag, 1991.

\bibitem{zbMATH02103273}
Allen Hatcher, \emph{Algebraic topology}, Cambridge: Cambridge University Press, 2002.

\bibitem{zbMATH04046721}
Mitsuyoshi Kato, \emph{On uniformizations of orbifolds}, Homotopy theory and related topics, {Pap}. {Symp}. {Kyoto}/{Jap}. 1984, {Adv}. {Stud}. {Pure} {Math}. 9, 149-172 (1987)., 1987.

\bibitem{loray-martin-09-MR2647972}
Frank Loray and David Mar{\'{\i}}n~P{\'e}rez, \emph{Projective structures and projective bundles over compact {Riemann} surfaces}, \'Equations diff\'erentielles et singularit\'es. En l'honneur de J. M. Aroca, Paris: Soci{\'e}t{\'e} Math{\'e}matique de France, 2009, pp.~223--252.

\bibitem{LORAY_2007}
Frank Loray and Jorge~Vit{\'o}rio Pereira, \emph{Transversely projective foliations on surfaces: existence of minimal form and prescription of monodromy}, Int. J. Math. \textbf{18} (2007), no.~6, 723--747.

\bibitem{loray2016representations}
Frank Loray, Jorge~Vit{\'o}rio Pereira, and Fr{\'e}d{\'e}ric Touzet, \emph{Representations of quasi-projective groups, flat connections and transversely projective foliations}, J. {\'E}c. Polytech., Math. \textbf{3} (2016), 263--308.

\bibitem{Scardua1997}
Bruno~Azevedo Sc{\'a}rdua, \emph{Transversely affine and transversely projective holomorphic foliations}, Ann. Sci. {\'E}c. Norm. Sup{\'e}r. (4) \textbf{30} (1997), no.~2, 169--204.

\bibitem{zbMATH07132579}
Calum Spicer, \emph{Higher-dimensional foliated {Mori} theory}, Compos. Math. \textbf{156} (2020), no.~1, 1--38.

\bibitem{uludag-07-MR2306159}
Abdurrahman~Muhammed Uluda{\u{g}}, \emph{Orbifolds and their uniformization}, Arithmetic and geometry around hypergeometric functions. Lecture notes of a CIMPA summer school held at Galatasaray University, Istanbul, Turkey, June 13--25, 2005, Basel: Birkh{\"a}user, 2007, pp.~373--406.

\bibitem{zbMATH01927232}
Claire Voisin, \emph{Th{\'e}orie de {Hodge} et g{\'e}om{\'e}trie alg{\'e}brique complexe}, Cours Sp{\'e}c. (Paris), vol.~10, Paris: Soci{\'e}t{\'e} Math{\'e}matique de France, 2002.

\end{thebibliography}
\end{document}